\documentclass[reqno]{amsart}

\usepackage{amsmath, amsthm}
\usepackage{amsfonts}
\usepackage{amssymb}
\usepackage[usenames,dvipsnames]{color} 
\usepackage{xypic}

\newcommand{\tn}[1]{\textnormal{#1}}

\newcommand{\m}{\mathfrak{m}}
\newcommand{\ds}{\displaystyle}
\newcommand{\z}{\mathbb{Z}}
\newcommand{\n}{\mathbb{N}}
\newcommand{\too}{\longrightarrow}
\newcommand{\onto}{\twoheadrightarrow}
\newcommand{\into}{\hookrightarrow}

\newcommand{\vp}{\varphi}

\newcommand{\ztt}{\z[t^{\pm 1}]}
\newcommand{\mcm}{maximal Cohen-Macaulay}
\newcommand{\krs}{Krull-Remak-Schmidt}
\newcommand{\fg}{finitely generated}
\newcommand{\modr}{\mathcal{C}(R)}
\newcommand{\modnor}{\mathcal{C}}
\newcommand{\mr}{\mathcal{M}(R)}
\newcommand{\mnor}{\mathcal{M}}

\newcommand{\bpr}{\begin{proof}}
\newcommand{\epr}{\end{proof}}

\DeclareMathOperator{\pd}{pd}

\DeclareMathOperator{\tor}{Tor}
\DeclareMathOperator{\ext}{Ext}

\DeclareMathOperator{\coker}{Coker}
\DeclareMathOperator{\rank}{rank}
\DeclareMathOperator{\depth}{depth}

\DeclareMathOperator{\ann}{Ann}
\renewcommand{\ker}{\tn{Ker}}

\renewcommand{\hom}{\tn{Hom}}

\theoremstyle{plain}
\newtheorem{thm}{Theorem}[section] 
\newtheorem{lem}[thm]{Lemma}
\newtheorem{prop}[thm]{Proposition}
\newtheorem{cor}[thm]{Corollary}
\theoremstyle{definition}
\newtheorem{defn}[thm]{Definition}
\newtheorem{ex}[thm]{Example}
\theoremstyle{remark}
\newtheorem{rmk}[thm]{Remark}

\newtheoremstyle{named}{}{}{\itshape}{}{\bfseries}{.}{.5em}{\thmnote{#3}}
\theoremstyle{named}
\newtheorem*{namedtheorem}{Theorem}

\numberwithin{equation}{section}

\definecolor{darkgreen}{RGB}{0,100,0}

\begin{document}

\title{Periodic modules over Gorenstein local rings}
\author{Amanda Croll}
\address{Department of Mathematics, University of Nebraska--Lincoln, Lincoln, NE 68588}
\email{s-acroll1@math.unl.edu}
\thanks{Research partially supported by NSF grants DMS-0903493 and DMS-1201889.}
\date{\today}

\subjclass[2010]{13D02, 13H10}

\begin{abstract}
It is proved that the minimal free resolution of a module $M$ over a Gorenstein local ring $R$ is eventually periodic if, and only if, the class of $M$ is torsion in a certain $\ztt$-module associated to $R$. This module, denoted $J(R)$, is the free $\ztt$-module on the isomorphism classes of finitely generated $R$-modules modulo relations reminiscent of those defining the Grothendieck group of $R$. The main result is a structure theorem for $J(R)$ when $R$ is a complete Gorenstein local ring; the link between periodicity and torsion stated above is a corollary. 
\end{abstract}

\maketitle



\section{Introduction}
This work makes a contribution to the study of eventually periodic modules over (commutative, Noetherian) local rings. In 1990, Avramov \cite{LA90} posed the problem of characterizing rings that have a periodic module. Eisenbud \cite{DE80} had previously shown that every complete intersection ring has a periodic module, but the question remained unanswered for other rings, even those which are Gorenstein. In Corollary \ref{tors_epmod}, we prove that a complete Gorenstein local ring $R$ has a periodic module if and only if there is torsion in a certain $\ztt$-module associated to $R$, where $\ztt$ denotes the ring of Laurent polynomials. This module, which we denote $J(R)$, is the free $\ztt$-module on the isomorphism classes of finitely generated $R$-modules modulo relations reminiscent of those defining the Grothendieck group of $R$; see Definition \ref{Jmod} and Proposition \ref{altdefn}. 

The main result of this paper is a structure theorem for $J(R)$ when $R$ is a Gorenstein local ring with the \krs{} property; see Theorem \ref{structurethm}. As a corollary, we deduce that an $R$-module is eventually periodic if and only if its class in $J(R)$ is annihilated by some non-zero element of $\ztt$. This leads to a characterization of hypersurface rings in terms of $J(R)$; see Corollary \ref{hypersurface}. 

This paper is motivated by work of D.R. Jordan \cite{DRJ10}, who defined the module $J(R)$ and proved that if the class of a module in $J(R)$ is torsion then the module has a rational Poincar\'{e} series. The converse, however, does not hold. Indeed, Jordan proved that, for an Artinian complete intersection ring $R$ with codimension at least two, the class of its residue field is not torsion in $J(R)$. Corollary \ref{hypersurface} contains this result, since the residue field of a complete intersection ring is eventually periodic if and only if the codimension is at most one.  

%
\section{The module $J(R)$} \label{Jmodsection}

All rings considered in this paper are commutative and Noetherian. Let $R$ be a ring and $\modr$ the set of isomorphism classes of \fg{} $R$-modules; write $[M]$ for the class of an $R$-module $M$ in $\modr$. When the ring is clear from context, we write $\modnor$ instead of $\modr$. 

\begin{defn}
\label{Jmod}
Let $F$ be the free $\ztt$-module $\ztt^{(\modnor)}$, that is, $$\ds F = \bigoplus_{[M]\in\modnor} \ztt[M],$$ and let $I$ be the $\ztt$-submodule generated by the following elements:
\begin{enumerate}
\item [(R1)] $[M]-[M']$ for every exact sequence of \fg{} $R$-modules\\ \makebox{$0\to P \to M \to M'\to 0$} with $P$ projective;
\item [(R2)] $[M]-t[M']$ for every exact sequence of \fg{} $R$-modules  \\ $0\to M' \to P \to M\to 0$ with $P$ projective;
\item [(R3)] $[M\oplus M'] - [M]-[M']$ for all \fg{} $R$-modules $M$ and $M'$.
\end{enumerate}
The main object of study in this article is the $\ztt$-module:
$$J(R) = F/I.$$
\end{defn}

In the following remark, we make a few observations about the module $J(R)$. 

\begin{rmk} 
\label{firstrmks} 
Let $M, M',$ and $P$ be \fg{} $R$-modules with $P$ projective.  
\begin{enumerate}
\item \label{rmk1}
$[P]=0$ in $J(R).$ 
\item \label{rmk2}
If \makebox{$0\to M \to M' \to P\to 0$} is exact, then $[M]-[M']=0$ in $J(R)$.
\end{enumerate}
\end{rmk}

Indeed, for (1), note that there is an exact sequence $0 \to P = P  \to 0$, and so the desired result follows from (R1).

To prove (2), notice that $M' \cong M\oplus P$ since $P$ is projective. Then in $J(R)$, $[M'] = [M] + [P]$ by (R3). Since $[P]=0$ in $J(R)$, it follows that $[M'] = [M]$.

The module $J(R)$ was defined by D.R. Jordan in \cite{DRJ10} and called the \emph{Grothendieck module}.  In Jordan's definition, the submodule $I$ is generated by four types of elements: the three given in Definition \ref{Jmod} as well as elements of the form $[M]-[M']$ where $M$ and $M'$ are modules as in Remark \ref{firstrmks}.(\ref{rmk2}). 

\begin{rmk} \label{gg}
The \emph{Grothendieck group} $\mathcal{G}$ of a ring $R$ is the free $\z$-module $\z^{(\mathcal{C})}$ modulo the subgroup generated by the Euler relations, that is, elements of the form \makebox{$[M']-[M]+[M'']$} for each exact sequence $0\to M'\to M\to M''\to 0$ of \fg{} \makebox{$R$-modules.} The \emph{reduced Grothendieck group}  $\overline{\mathcal{G}}$ of $R$ is the group $\mathcal{G}$ modulo the subgroup generated by classes of modules of finite projective dimension. We note that $\overline{\mathcal{G}} = J(R)/L$, where $L$ is the submodule generated by the Euler relations. 
\end{rmk}
 
\subsection{Syzygies}
In order to discuss syzygies, we recall Schanuel's Lemma; a proof can be found in \cite[Thm 4.1.A]{IK74}.
\begin{namedtheorem}[Schanuel's Lemma]\label{Schanuel}
Given exact sequences of $R$-modules
$$ 0\to K\to P\to M \to 0 \quad \tn{ and } \quad 0\to K'\to P'\to M \to 0$$
with $P$ and $P'$ projective, there is an isomorphism $K \oplus P' \cong K' \oplus P$ of \makebox{$R$-modules. \qedsymbol}
\end{namedtheorem}

Let $R$ be a ring and $M$ an $R$-module. Denote by $\Omega_R M$ any $R$-module that is the kernel of a homomorphism of $R$-modules $P\twoheadrightarrow M$ with $P$ a finitely generated projective. While $\Omega_R M$ depends on the choice of $P$, Schanuel's Lemma shows that $M$ determines $\Omega_R M$ up to a projective summand. Any module isomorphic to a module $\Omega_R M$ is called a \emph{syzygy} of $M$. For any $d>1$, a \emph{dth syzygy} of $M$ is a module  $\Omega^d_R M$ such that \makebox{$\Omega^d_R M = \Omega_R (\Omega^{d-1}_R M)$} for some $(d-1)st$ syzygy of $M$. By Schanuel's Lemma, $\Omega^d _R M$ is also determined by $M$ up to a projective summand. For any $n\ge 0$, we write $\Omega^n M$ when the ring is clear from context. 

The syzygy gives a well-defined functor on $J(R)$, as shown in Lemma \ref{syzy_hom}. The following remark will aid in this discussion. 
\begin{rmk} \label{bigdiag}
If $0\to P\to M \to M'\to 0$ is an exact sequence of $R$-modules with $P$ projective, then there is a module that is a syzygy of both $M$ and $M'$.

Indeed, pick a surjective map $G'\twoheadrightarrow M'$, with $G'$ a projective $R$-module.  Consider the following diagram. Let $X$ be the pullback of $M\to M'$ and $G'\to M'$. 
Since $G'\to M'$ is surjective, $X\to M$ is also surjective. Since $G'$ and $P$ are projective, $X$ is projective. Hence the kernel of $X\onto M$ is a syzygy of $M$; let $N$ be this kernel. Let $N'$ denote the kernel of $G\onto M'$. Then there is a commutative diagram with exact rows as follows.

\begin{center} \hspace{0pt}
\xymatrix{
0\ar[r] & P \ar[r] & M \ar[r] & M' \ar[r] & 0\\
0 \ar[r] & P\ar@{=}[u] \ar[r] & X\ar@{->>}[u] \ar[r] & G'\ar@{->>}[u] \ar[r]  & 0\\
 & & N \ar@{^{(}->}[u]\ar[r]^{\cong} & N' \ar@{^{(}->}[u]
}
\end{center}
This justifies the claim.
\end{rmk}

\begin{lem} \label{syzy_hom}
Assigning $[M]$ to $[\Omega M]$ induces a $\ztt$-linear map $$\Omega\!:\!J(R) \to J(R).$$
\end{lem}

\bpr
By Schanuel's Lemma, the assignment $[M] \mapsto [\Omega M]$ gives a homomorphism $$\ds \tilde{\Omega}: \bigoplus_{[M]\in \modnor} \ztt [M] \to J(R)$$  of $\ztt$-modules. It is enough to check that (R1), (R2), and (R3) from Definition \ref{Jmod} are in $\ker(\tilde{\Omega}),$ so that $\tilde{\Omega}$ factors through $J(R)$; the induced map is $\Omega$.

For (R3), note that for any syzygies $\Omega M$ of $M$ and $\Omega M'$ of $M'$, the $R$-module $\Omega M \oplus \Omega M'$ is a syzygy of $M\oplus M'$. Since $[\Omega M \oplus \Omega M'] = [\Omega M] + [\Omega M']$ in $J(R)$, one finds that 
$\tilde{\Omega}([M\oplus M']) = \tilde{\Omega}([M])\oplus \tilde{\Omega} ([M'])$.

Next, we consider (R2). Given an exact sequence of \fg{} $R$-modules $0\to M'\to P\to M\to 0$ with $P$ projective, we show that $[M']-t^{-1}[M]$ is in $\ker(\tilde{\Omega})$. In $J(R)$, one has $\tilde{\Omega}([M']) = t^{-1}[M']$.
Since $M'$ is a syzygy of $M$, we have $[M'] = [\Omega M]$ in $J(R)$. 
As $[\Omega M] = \tilde{\Omega}([M])$, the $\ztt$-linearity of $\tilde{\Omega}$ implies that 
$$\tilde{\Omega}([M']) = t^{-1}\tilde{\Omega}([M]) = \tilde{\Omega}(t^{-1}[M]).$$
Therefore (R2) is in $\ker(\tilde{\Omega})$.

Finally, we verify that (R1) is in $\ker (\tilde{\Omega})$. Given an exact sequence of \fg{} $R$-modules  $0\to P\to M \to M'\to 0$ with $P$ projective, Remark \ref{bigdiag} implies there is a module $L$ that is a syzygy of both $M$ and $M'$. Therefore \makebox{$\tilde{\Omega}([M]) = [L] = \tilde{\Omega}([M'])$.}
\epr
 
For $i>1$, define $\Omega^i:  J(R) \to J(R)$ by $\Omega^i = \Omega\circ \Omega^{i-1}.$ The next remark demonstrates a relationship in $J(R)$ between the class of a module and the classes of its syzygies.

\begin{rmk} \label{syzygy}
Let $M$ be a finitely generated $R$-module. Then $[M] = t^n[\Omega^n M]$ in $J(R)$ for any $n \in \n$. 

Indeed, (R2) implies that $t[\Omega M] =  [M]$ in $J(R)$. Iterating this, one finds that $[M] = t^n[\Omega^n M]$ for all $n\in\n$.
\end{rmk}

In the next proposition, we give an alternate description of $J(R)$ which makes the relations in this module more transparent. 

\begin{prop} \label{altdefn}
Let $F$ be the free $\ztt$-module $\ztt^{(\mathcal{C})}$, and let $L$ be the \makebox{$\ztt$-submodule} generated by the following elements:
\begin{itemize}
\item [(R1$'$)] $[P]$ for every \fg{} projective $R$-module $P$;
\item [(R2$'$)] $[M]-t[\Omega M]$ for every \fg{} $R$-module $M$;
\item [(R3)] $[M\oplus M'] - [M]-[M']$ for all \fg{} $R$-modules $M$ and $M'$.
\end{itemize}
There is an isomorphism of $\ztt$-modules 
$$J(R) \cong F/L.$$
\end{prop}

\bpr
Via a proof similar to the proof of Lemma \ref{syzy_hom}, it can be shown that assigning $[M]$ to $[\Omega M]$ induces a $\ztt$-linear map \makebox{$\Omega: F/L \to F/L$}

Let $\tilde{q}: F\to J(R)$ be the quotient map. We show that (R1$'$), (R2$'$), and (R3) are in $\ker(\tilde{q})$, and hence $\tilde{q}$ factors through the quotient $F/L$ via a map $q: F/L\to J(R)$. 

The elements given by (R1$'$) are in $\ker(\tilde{q})$ by Remark \ref{firstrmks}.(\ref{rmk1}), and those from (R2$'$) are in $\ker(\tilde{q})$ by Remark \ref{syzygy}. The elements given by (R3) are in $\ker(\tilde{q})$ by the definition of $J(R)$.

Let $\tilde{p}: F\to F/L$ be the quotient map. We show that (R1), (R2), and (R3) are in $\ker(\tilde{p})$, and hence $\tilde{p}$ factors through the quotient $J(R)$ by a map $p: J(R)\to F/L$. Note that the elements given by (R3) are in $\ker(\tilde{p})$ by the definition of $L$. It remains to verify that (R1) and (R2) are in $\ker(\tilde{p})$. 

First, consider (R1). Let $0\to P \to M \to M' \to 0$ be an exact sequence of \fg{} $R$-modules with $P$ projective. 
By (R2$'$), one has $$[M]-[M'] = [M]-t[\Omega M']$$ in $F/L$ for any syzygy $\Omega M'$ of $M'$.
Given a syzygy $\Omega M$ of $M$, Remark \ref{bigdiag} shows that there is a projective \makebox{$R$-module} $G$ such that $\Omega M \oplus G$ is a syzygy of $M'$. 
Hence $[\Omega M'] = [\Omega M \oplus G] = [\Omega M]$ in $F/L$, and thus 
\[[M]-[M'] = [M]-t[\Omega M]=0\]
in $F/L$. Therefore (R1) is in $\ker(\tilde{p})$. 

Finally, we show that (R2) is in $\ker(\tilde{p})$. Let $0\to M'\to P\to M\to 0$ be an exact sequence of \fg{} $R$-modules. Then $M'$ is a syzygy of $M$, so 
$t[M']-[M] = t[\Omega M]-[M]=0$ by (R2$'$). Hence (R2) is in $\ker(\tilde{p})$.

Note that $p\circ q$ is the identity map on $F/L$; thus $p$ is injective. Since $p$ is a quotient map and hence also surjective, $p$ is an isomorphism.
\epr

Recall that a homomorphism of rings $\vp:R\to S$ is \emph{flat} if $S$ is flat as an $R$-module via $\vp$. A straightforward argument yields the following result.
\begin{lem} \label{flat}
Let $\vp: R\to S$ be a homomorphism of rings. When $\vp$ is flat,  the assignment $[M] \mapsto [S\otimes_R M]$ induces a homomorphism of $\ztt$-modules 
\[
\pushQED{\qed} 
J(\vp): J(R) \to J(S). \qedhere
\popQED
\]

\end{lem}

\newcommand{\nm}{\mathfrak{n}}

\subsection{Finite projective dimension}
In \cite[Prop 3]{DRJ10} Jordan proves the following: if $R$ is a commutative local Noetherian ring and $M$ a finitely generated $R$-module, then the projective dimension of $M$ is finite if and only if $[M]=0$ in $J(R)$.  In Proposition \ref{finprojdim}, we extend this result to all commutative Noetherian rings.

\begin{defn}
Let $(R,\m,k)$ be a local ring with maximal ideal $\m$ and residue field $k$, and let $M$ be a \fg{} $R$-module. Set \[ \beta_i(M)= \rank_k \tor_i^R(M,k);\]
this is the \emph{$i$th Betti number} of $M$. The \textit{Poincar\'e series} of $M$ is given by \[\ds P_M^R(t) = \sum_{i=0}^\infty \beta_i(M)\,t^i\] viewed as an element in the formal power series ring $\ztt$. 
\end{defn}

Let $\z((t))$ denote the \textit{ring of formal Laurent series}, $\z[[t]]\!\left[\frac{1}{t}\right]$; we view it as a module over $\ztt$. Notice that $\ztt$ is a $\ztt$-submodule of $\z((t))$.

The following proposition is \cite[Lem 1]{DRJ10}. We include the statement here for ease of reference.

\begin{prop} \label{homompoin}
Let $R$ be a local ring. The assignment $[M]\mapsto P_M^R(t)$ induces a homomorphism of $\z[t^{\pm1}]$-modules 
\[
\pushQED{\qed} 
\pi: J(R)\to \z((t))/\ztt. \qedhere
\popQED
\] 
\end{prop}

\begin{defn}
An $R$-module $M$ has \textit{finite projective dimension} if an $i$th syzygy module $\Omega^ i M$ is projective for some $i\ge 0$; in this case, we write $\pd_R M <\infty$.
\end{defn}

By Schanuel's Lemma, an $i$th syzygy module is projective if and only if every $i$th syzygy module is projective. Observe that if $\Omega^i M$ is projective, then $\Omega^j M$ is projective for all $j\ge i$. When $R$ is local, an $R$-module $M$ has finite projective dimension if and only if $\beta_i(M)=0$ for $i\gg 0$; see \cite[Cor 1.3.2]{BH93}.

The following proposition was proved in \cite[Prop 3]{DRJ10} for local rings.

\begin{prop} \label{finprojdim}
Let $R$ be a commutative Noetherian ring and $M$ a finitely generated $R$-module. Then $[M]=0$ in $J(R)$ if and only if the projective dimension of $M$ is finite.
\end{prop}

\bpr  
If the projective dimension of $M$ is finite, then $[\Omega^n M] = 0$ for some $n\in\n$.  Hence $[M]=0$ by Remark \ref{syzygy}. 

Suppose $[M] = 0$ in $J(R)$. First, we consider the case when $R$ is local. Using the homomorphism $\pi$ from Proposition \ref{homompoin}, one finds that $P_R(M)\in \ztt$. Hence $P_R(M)$ is a polynomial, and it follows that $\beta_i(M) = 0$ for $i \gg 0$. Thus the projective dimension of $M$ is finite. 

For a general ring $R$, the map $R\to R_\m$ is flat for each maximal ideal $\mathfrak{m}.$  Lemma \ref{flat} gives a homomorphism $J(R) \to J(R_\mathfrak{m})$ with  $[M] \mapsto [M_\mathfrak{m}]$. Thus $[M_\mathfrak{m}]=0$ in $J(R_\mathfrak{m})$, and hence $\pd_{R_\mathfrak{m}} M_\mathfrak{m} < \infty$. Hence the projective dimension of $M$ over $R$ is finite by \cite[Thm 4.5]{BM67}.
\epr


\section{MCM modules over Gorenstein local rings} \label{gorrings1}

In this section, we collect known, but hard to document, properties of MCM modules over Gorenstein local rings.

For the remainder of this article, let $R$ be a local ring with residue field $k$ and $M$ a \fg{} $R$-module. Set $(-)^*=\hom_R(-,R)$. 

A \emph{free cover} of $M$ \cite[Def 5.1.1]{EJ00} is a homomorphism $\vp: G \to M$ with $G$ a free $R$-module such that 
\begin{enumerate}
\item for any homomorphism $g: G'\to M$ with $G'$ free there exists a homomorphism $f: G'\to G$ such that $g = \vp f$, and 
\item any endomorphism $f$ of $G$ with $\vp = \vp f$ is an automorphism.
\end{enumerate}
A free cover is unique up to isomorphism.

Let $\nu_R(M)$ denote the minimal number of generators of an $R$-module $M$, i.e., $\nu_R(M) = \rank_k (k\otimes_R M)$.

\begin{rmk} \label{coverexists}
Every $R$-module admits a free cover. A homomorphism $\vp: R^n \to M$ is a free cover of $M$ if and only if $\vp$ is surjective and $n=\nu_R(M)$.
\end{rmk}

A \emph{free envelope} of $M$ \cite[Def 6.1.1]{EJ00} is a homomorphism $\vp: M \to G$ with $G$ a free $R$-module such that
\begin{enumerate}
\item for any homomorphism $g: M\to G'$ with $G'$ free there exists a homomorphism $f: G\to G'$ such that $g = f\vp$, and 
\item any endomorphism $f$ of $G$ with $\vp = f \vp$ is an automorphism of $G$.
\end{enumerate}
A free envelope is unique up to isomorphism. 

\begin{rmk} \label{freeenv} Every \fg{} $R$-module $M$ admits a free envelope. Indeed, the homomorphism $f = (f_1,\dots, f_n): M\to R^n$, where $f_1,\dots,f_n$ is a minimal system of generators of $M^*$, is a free envelope of $M$. 

The free envelope of $M$ can also be constructed as follows. Let $R^n\onto M^*$ be the free cover of $M^*$. Applying $(-)^*$ to this map, one has an injection $M^{**}\to R^n$. The composite map $M\to M^{**} \to R^n$ is the free envelope of $M$, where $M\to M^{**}$ is the natural biduality map.\end{rmk}

\begin{rmk}
For a local ring $R$, one can choose $\Omega M$ so that it is unique up to isomorphism by selecting $\Omega M = \ker(\vp)$ for a free cover $\vp$ of $M.$ Hence from this section on, $\Omega(-)$ is well-defined, up to isomorphism, on the category of $R$-modules. 
\end{rmk}

\begin{defn}
The \emph{cosyzygy module} of $M$ is $\Omega^{-1}_R M = \coker (\vp)$, where $\vp$ is the free envelope of $M$. For $n>1$, the \textit{nth cosyzygy module} of $M$ is $$\Omega^{-n}_R M = \Omega^{-1}_R (\Omega^{-(n-1)}_R M).$$
\end{defn}

In \cite[Sect 8.1]{EJ00}, the authors refer to the cosyzygy module as the \emph{free cosyzygy module}; since this is the only cosyzygy module studied in this article, we simply call it the cosyzygy module. We note that, when the module $M$ is torsion-free, the cosyzygy module is also called the \emph{pushforward}; see \cite{HJW01}. 

\subsection{Maximal Cohen-Macaulay modules}
A non-zero \makebox{$R$-module} $M$ is said to be \textit{\mcm{}} (abbreviated to MCM) if  $\depth_R M = \dim R$. 

If $R$ is Cohen-Macaulay, the set of isomorphism classes of MCM, non-free, indecomposable modules generates $J(R)$ as a module over $\ztt$ since \makebox{\cite[Prop 1.2.9]{BH93}} implies that each $R$-module has a syzygy that is either MCM or zero. 
If $R$ is Gorenstein and has the Krull-Remak-Schmidt property, one can do better: $J(R)$ is generated over $\z$ by the isomorphism classes of MCM, non-free, indecomposable modules; see Theorem \ref{structurethm}. 

The ring $R$ is \emph{Gorenstein} if it has finite injective dimension as a module over itself. Equivalently, $R$ is Gorenstein provided it is Cohen-Macaulay and $\ext^i_R(M,R)=0$ for all MCM modules $M$ and all $i\ge 1;$ this equivalence can be seen from \cite[Satz 2.6]{FI69} and \makebox{\cite[Prop 3.1.10]{BH93}.} 

For the remainder of this article, we focus on Gorenstein rings. The following are well-known results on MCM modules that will be used throughout the paper; for lack of adequate references, some of the proofs are given here. 

\begin{rmk} \label{gormcmrmks} Let $R$ be a Gorenstein local ring and $M$ an MCM $R$-module. 
\begin{enumerate}
\item \label{dMCM} 
Let $N$ be an $R$-module. If $d\ge \dim R$, then $\Omega^d N$ is MCM or zero.

\item \label{doubledual}
The natural homomorphism $M\to M^{**}$ is an isomorphism.

\item \label{freeenvinj}
The free envelope of $M$ is an injective homomorphism.

\item \label{cosyzMCM}
The modules $\Omega M$ and $\Omega^{-1}M$ are MCM.

\item \label{syzdual} 
$(\Omega^{-1}M)^* \cong \Omega (M^*)$.

\item \label{indec}
If $M$ is indecomposable, then $\Omega M$ and $\Omega^{-1}M$ are also indecomposable.

\item \label{nofreesmds} 
If $M$ has no free summands, then the modules $\Omega M$ and $\Omega^{-1}M$ also have no free summands. 

\item \label{ddiso} If $M$ has no free summands, then $\Omega^{-n} \Omega^n M \cong M$ for all $n\in\z$. 
\end{enumerate}
\end{rmk}

Property (\ref{dMCM}) follows from the Depth Lemma \cite[Prop 1.2.9]{BH93}. Property (\ref{doubledual}) is proved in \cite[Cor 2.3]{WV68}. For (\ref{cosyzMCM}), a proof that $\Omega M$ is an MCM module is given in \makebox{\cite[Lem 1.3]{JH78}} and \cite[Prop 1.6.(2)]{HJW01} shows that $\Omega^{-1}M$ is MCM. 

\bpr [Proof of \emph{(\ref{freeenvinj})}] 
The free envelope of $M$ is the composition
$$
M \to M^{**} \into F^{*},
$$
where $F\onto M^*$ is the free cover of $M^*$. So (\ref{freeenvinj}) follows from (\ref{doubledual}).
\renewcommand{\qedsymbol}{}
\epr 

\bpr[Proof of \emph{(\ref{syzdual})}]
Let $\pi: F\to M^*$ be the free cover of $M^*$. Since the natural map $M\to M^{**}$ is an isomorphism, $\pi^*: M \to F^*$ is the free envelope of $M$ by Remark \ref{freeenv}. Thus $\Omega^{-1}M$ is defined by an exact sequence 
$$0\too M \stackrel{\pi^*}{\too} F^* \too \Omega^{-1} M \too 0.$$
Applying $(-)^*$ to this sequence yields the exact sequence
$$0\too (\Omega^{-1} M)^* \too F\stackrel{\pi}{\too} M^* \too 0.$$
As $\pi$ is the free cover of $M^*$, one gets $\Omega (M^*) \cong (\Omega^{-1}M)^*.$
\renewcommand{\qedsymbol}{}
\epr

\bpr[Proof of \emph{(\ref{indec})}]
A proof that $\Omega M$ is indecomposable is given in \cite[Lem 1.3]{JH78}. We prove that $\Omega^{-1}M$ is indecomposable. Let $G$ be the free envelope of $M$. By \eqref{freeenvinj}, the following sequence is exact:   
$$0\to M\to G\to \Omega^{-1}M\to 0.$$
Since $\Omega^{-1} M$ is MCM, applying $(-)^*$ to this sequence yields the exact sequence
$$0\to (\Omega^{-1}M)^*\to G^*\to M^*\to 0.$$
By (\ref{syzdual}), $(\Omega^{-1}M)^* \cong \Omega (M^*).$ Since $M$ is indecomposable and isomorphic to $M^{**}$, it follows that $M^*$ is indecomposable. As $M^*$ is an indecomposable MCM module, $(\Omega^{-1}M)^*$ is indecomposable by the result for syzygies. Thus $\Omega^{-1}M$ is also indecomposable.
\renewcommand{\qedsymbol}{}
\epr

\bpr[Proof of \emph{(\ref{nofreesmds})}]
Suppose $\Omega M \cong N\oplus R$. Let $G$ be the free cover of $M$, and let $X$ be the pushout of $N\oplus R \onto R$ and $N\oplus R \to G$. Then we have the following commutative diagram with exact rows.

\begin{center} \hspace{0pt}
\xymatrix{
0 \ar[r] & N\oplus R \ar[r] \ar@{->>}[d] & G \ar[r]  \ar@{->>}[d] & M \ar@{=}[d] \ar[r] & 0\\
0 \ar[r] & R \ar[r] & X \ar[r] & M \ar[r] &0
}
\end{center}
Since $M$ is MCM one has $\ext_R^1(M,R)=0$, so $X\cong M\oplus R$.
Then $\nu_R(G) \ge \nu_R(M) + 1$ as $G$ maps onto $M\oplus R$. However, this is a contradiction since $G$ is the free cover of $M$. Hence $\Omega M$ has no free summand.

Since $M$ has no free summand, $M^*$ has no free summand. By property (\ref{syzdual}) and the result for syzygies, $(\Omega^{-1}M)^*$ has no free summand. 
Hence $\Omega^{-1}M$ also has no free summand. 
\renewcommand{\qedsymbol}{}
\epr

\bpr[Proof of \emph{(\ref{ddiso})}]
First, note that $M\cong \Omega (\Omega^{-1} M) \oplus F'$ for some free module $F'$ by Schanuel's Lemma. Since $M$ has no free summands, $M \cong \Omega (\Omega^{-1} M)$. 
Next, we show that $\Omega^{-1}(\Omega M) \cong M$; the result then follows by induction on $n$. 

Let $\Omega M\to G$ be the free envelope of $\Omega M$, and let $G'\to M$ be the free cover of $M$. We have the following commutative diagram.

\begin{center} \hspace{0pt}
\xymatrixrowsep{.3in}
\xymatrixcolsep{.3in}
\xymatrix{
0\ar[r] & \Omega M \ar[r]^{i} & G \ar[r] & \Omega^{-1} (\Omega M) \ar[r] & 0\\
0 \ar[r] & \Omega M \ar@{=}[u] \ar[r]^{j} & G'\ar@{->>}[u]_{f} \ar[r] & M \ar@{->>}[u] \ar[r]  & 0\\
 & &K \ar@{^{(}->}[u]\ar[r]^{\cong} & K \ar@{^{(}->}[u]}
\end{center}
Indeed, since $M$ is MCM there are maps $f: G'\to G$ and $g:G\to G'$ such that $f \circ j = i$ and $g \circ i = j$. Then $i = i \circ (f\circ g)$, and $f\circ g$ is an isomorphism since $i$ is the free envelope of $\Omega M$. Hence $f: G'\to G$ is surjective. Thus the map $M\to \Omega^{-1}(\Omega M)$ is also surjective. Note that the kernels of $G'\onto G$ and $M\onto \Omega^{-1}(\Omega M)$ are isomorphic by the Snake Lemma. Since $\Omega^{-1}(\Omega M)$ is MCM and $K$ is free,  $M\cong \Omega^{-1} (\Omega M) \oplus K$.
Since $M$ has no free summands, $M\cong \Omega^{-1}(\Omega M).$
\epr

\begin{prop} \label{Jmodmcm} Let $R$ be a Gorenstein local ring and $M$ a finitely generated \makebox{$R$-module.}\begin{enumerate}
\item \label{cosyzeq} If $M$ is an MCM module, then $t^{-n}[\Omega^{-n} M] = [M]$ in $J(R)$ for each $n\in\z$. 
\item \label{CosyzOfSyz} There is an MCM $R$-module $N$ with $[M]=[N]$ in $J(R)$. 
\end{enumerate}
\end{prop}

\bpr
Remark \ref{syzygy} showed (1) for $n\le 0$. Using Remark \ref{gormcmrmks}.(\ref{freeenvinj}), a proof similar to that of Remark \ref{syzygy} yields the desired result.

For (2), let $d=\dim R$.  By Remark \ref{syzygy}, $t^d[\Omega^{d} M] =[M].$ By (\ref{cosyzeq}), $$t^{-d}[\Omega^{-d}\Omega^d M] = [\Omega^d M].$$  Hence $[M] = [\Omega^{-d}\Omega^d M]$.  
\epr

\section{Gorenstein local rings: structure of $J(R)$}

The main result of this section is a structure theorem for $J(R)$ when $R$ is a Gorenstein local ring; see Theorem \ref{structurethm}. 


Throughout this section $R$ will be a Gorenstein local ring. Recall that $\Omega^n(-)$ denotes the $n$th (co)syzygy module, which is well-defined up to isomorphism.

Recall that a local ring $R$ is said to have the \textit{Krull-Remak-Schmidt property} if the following condition holds:  given an isomorphism of \fg{} $R$-modules $$\ds \bigoplus_{i=1}^m M_i \,\, \cong \,\,\bigoplus_{j=1}^n N_j$$ 
where $M_i$ and $N_j$ are indecomposable and non-zero, $m=n$ and, after renumbering if necessary, $M_i \cong N_i$ for each $i.$

Henselian local rings, and in particular complete local rings, have the Krull-Remak-Schmidt property; see \cite[Thm 1.8]{LW12} and \cite[Cor 1.9]{LW12}.

\begin{rmk}\label{ZG}
In order to set up notation for the next theorem, we first discuss a special type of \makebox{$\ztt$-module.} For this, we view $\ztt$ as the group algebra  over $\z$ of the free group $G=\langle t \rangle$ on a single generator $t$; that is, $G\cong (\z,+)$. Let $X$ be a set with a $G$-action. Let $\z X = \z^{(X)},$ the free $\z$-module with basis given by the elements of $X$, and let $\z G$ be the group algebra over $\z$ of $G$. Then $\z X$ is naturally a $\z G$-module \cite[Ch.III, \S1]{SL02}.
\end{rmk}

In what follows, we let 
\[
\mr =
\left\{[M]\in \modr
 \left|
\begin{gathered}
\text{$M$ is MCM, non-free,}\\
\text{ and indecomposable}
\end{gathered}
\right.\right\}.
 \]
When $R$ is clear from context, we write $\mathcal{M}$ for $\mr$. 


Remark \ref{gormcmrmks} properties (\ref{cosyzMCM}), (\ref{indec}), and (\ref{nofreesmds}) imply that $[\Omega M]$ and $[\Omega^{-1} M]$ are in $\mnor$ if \makebox{$[M]\in\mnor$.} Thus there is an action of $G$ on $\mnor$ with $t[M] = [\Omega^{-1} M]$ and \makebox{$t^{-1}[M] = [\Omega M]$.} Let $\mathcal A = \z^{(\mathcal{M})}$ be the corresponding $\ztt$-module. The canonical map $\mathcal{M} \to J(R)$ induces a $\ztt$-linear homomorphism: \[\ds \Phi: \mathcal{A} \to J(R).\] 
Assume $R$ has the Krull-Remak-Schmidt property. We define a $\z[t^{\pm1}]$-linear homorphism
\begin{equation} \label{psi}
\psi: \bigoplus_{[M] \in \modnor}  \ztt[M] \hspace{.1in} \too \mathcal{A}
\end{equation}
by setting $\psi([M]) = \sum_{i=1}^n [M_i],$ where each $M_i$ is indecomposable and 
$$\Omega^{-(d+1)}\Omega^{d+1} M \cong \bigoplus_{i=1}^n M_i$$
with $d = \dim R$.
Since $\Omega^{-(d+1)}\Omega^{d+1} M$ is either zero or MCM with no free summands,
$ \sum_{i=1}^n [M_i]$ is indeed in 
$\mathcal{A}$.
Since $R$ has the \krs{} property, $\psi$ is well-defined. 

\begin{thm}\label{structurethm}
Let $R$ be a Gorenstein local ring that has the \krs{} property. 
Then the $\ztt$-linear map
$$\Phi: \mathcal{A} \to J(R)$$
is an isomorphism with inverse $\Psi$ induced by $\psi$ described in \eqref{psi}.
\end{thm}
 
\bpr
To show that $\psi$ induces a $\ztt$-linear map $\Psi: J(R) \to \mathcal{A},$
it suffices to show that the elements described in (R1), (R2), and (R3) from Definition \ref{Jmod} are in the kernel of $\psi$. 

For elements given by (R3), note that 
$$\Omega^{-(d+1)}(\Omega^{d+1} (M\oplus N)) \cong \Omega^{-(d+1)}\Omega^{d+1} M  \oplus \Omega^{-(d+1)}\Omega^{d+1} N.$$
Then $\psi([M\oplus N]) = \psi([M]) + \psi([N])$, and hence (R3) is in $\ker( \psi)$. 

Next, we consider (R2): given an exact sequence $0\to M'\to P\to M\to 0$ of finitely generated $R$-modules with $P$ projective, we show that $\psi([M'])=\psi(t^{-1}[M])$. By Schanuel's Lemma there exists a free $R$-module $G$ such that \makebox{$M'\cong \Omega M \oplus G$.} Since $\psi([G]) = 0$, we have \makebox{$\psi([M']) = \psi([\Omega M])$} in $J(R)$. Next, we show that \makebox{$\psi([\Omega M]) = \psi(t^{-1}[M])$.} By Remark \ref{gormcmrmks}.(\ref{ddiso}), 
$$\Omega^{-(d+1)}\Omega^{d+1}(\Omega M) \cong \Omega( \Omega^{-(d+2)}\Omega^{d+2} M) \cong \Omega (\Omega^{-(d+1)}\Omega^{d+1} M).$$
Note that $\Omega^{-(d+1)}\Omega^{d+1}(\Omega M)$ determines $\psi([\Omega M])$ and $\Omega (\Omega^{-(d+1)}\Omega^{d+1} M)$ determines $\psi(t^{-1}[M])$. Hence $\psi([\Omega M]) = \psi(t^{-1}[M])$, and therefore (R2) is in $\ker (\psi)$. 

It remains to verify that (R1) is in $\ker( \psi).$ Let $0\to P \to M \to M'\to 0$ be an exact sequence of $R$-modules with $P$ projective. By Remark \ref{bigdiag}, there are free $R$-modules $G$ and $G'$ such that $\Omega M \oplus G \cong \Omega M' \oplus G'$. 
Since $\psi([G]) =\psi([G'])=0$, $\psi([\Omega M] = \psi([\Omega M'])$. As (R2) is in $\ker(\psi)$, one finds that $\psi(t^{-1}[M]) = \psi(t^{-1}[M'])$ and thus $\psi([M]) = \psi([M']).$ Hence (R1) is in $\ker(\psi)$. 

Thus $\psi$ factors through the quotient $J(R)$ via a homomorphism 
$\Psi: J(R) \to \mathcal{A}.$ Notice that $\Psi \circ \Phi$ is the identity.
Indeed, if $M$ is an MCM module with no free summands, then $\Omega^{-(d+1)}\Omega^{d+1} M\cong M$ by Remark \ref{gormcmrmks}.(\ref{ddiso}). Hence $\Phi$ is injective. For each $R$-module $N$, Proposition \ref{Jmodmcm}.(\ref{CosyzOfSyz}) shows that there is an MCM $R$-module $M$ such that $[M]=[N]$. Thus $\Phi$ is also surjective and hence an isomorphism.
\epr
 
\begin{rmk}
If $R$ is Gorenstein, Theorem \ref{structurethm} implies that $J(R)$ is torsion-free as an abelian group. We do not know whether this holds for a general local ring $R$.
\end{rmk} 

In \cite[Lem 8]{DRJ10}, the following result is proved for Artinian Gorenstein rings. 
 
\begin{cor} \label{lemma8}
Let $R$ be a Gorenstein local ring, and let $M$ and $N$ be \fg{} MCM $R$-modules. Then $[M] = [N]$ in $J(R)$ if and only if $$M\oplus R^m \cong N \oplus R^n$$ for some $m,n\in \z_{\ge 0}$. Thus if neither $M$ nor $N$ has a free summand, $[M]=[N]$ in $J(R)$ if and only if $M \cong N$. 

\end{cor}

\begin{proof} 
If $M\oplus R^m \cong N \oplus R^n$ for some $m, n\in\z_{\ge 0}$, then  in $J(R)$ we have 
\[[M] = [M\oplus R^m] = [N\oplus R^n] = [N].\]

Suppose that $[M] = [N]$ in $J(R)$. We may assume $M$ and $N$ have no free summands. 
We first prove the result under the assumption that $R$ is complete with respect to the maximal ideal. Complete rings have the \krs{} property for \fg{} modules; see for example \cite[Cor 1.10]{LW12}. Hence Theorem \ref{structurethm} applies.
Let $d = \dim R$, and let \makebox{$\ds \Psi: J(R) \to \mathcal{A}$} be the isomorphism given in Theorem \ref{structurethm}. Suppose
$$M  = \hspace{-5pt}\bigoplus_{[M_\lambda]\in \mnor} \hspace{-5pt} M_\lambda^{e_\lambda} \quad \tn{ and } \quad N = \hspace{-5pt}\bigoplus_{[M_\lambda]\in \mnor} \hspace{-5pt} M_\lambda^{f_\lambda}$$ 
where $e_\lambda, f_\lambda\ge 0$. 
From Remark \ref{gormcmrmks}.(\ref{nofreesmds}) and the definition of $\psi$ given in \eqref{psi}, one gets an equality $$ \sum_{[M_\lambda] \in \mnor} \hspace{-4pt} e_\lambda [M_\lambda] = \Psi([M]) = \Psi([N]) = \hspace{-6pt} \sum_{[M_\lambda] \in \mnor} \hspace{-4pt}  f_\lambda [M_\lambda].$$
Since $\mathcal{A}$ is free on $\mathcal{M}$, we have $e_\lambda = f_\lambda$ for all $\lambda$. Therefore $M \cong N$ as $R$-modules. 

Now suppose that $R$ is any local ring with maximal ideal $\mathfrak{m}$. Write $\widehat{R}$ for the $\mathfrak{m}$-adic completion of $R$. If $[M]=[N]$ in $J(R)$, then $[M\otimes_R \widehat{R}]=[N\otimes_R \widehat{R}]$ in $J(\widehat{R})$ by Lemma \ref{flat}. 

Note that an $R$-module $M$ has a free summand if and only if the evaluation map $ev: M^*\otimes _R M \to R$, where $\vp\otimes m \mapsto \vp(m)$, is surjective.
But if this map is surjective for $M$, then the map $ev\otimes_R \widehat{R}$ is also surjective. So since $M$ and $N$ have no free summands, $M\otimes_R \widehat{R}$ and $N\otimes_R \widehat{R}$ also have no free summands.  

The result for complete rings then shows that $M\otimes_R \widehat{R} \cong N\otimes_R \widehat{R}$ as $\widehat{R}$-modules, and \cite[Cor 1.15]{LW12} implies that $M\cong N$.

Note that cancellation of direct summands is valid over local rings \cite[Cor 1.16]{LW12}. Then $M\oplus R^m \cong N \oplus R^n$ implies that $M\oplus R^{m'} \cong N$ or $M \cong N \oplus R^{n'}$. Thus if neither $M$ nor $N$ has a free summand, $M\cong N$. 
\end{proof}

\section{Gorenstein local rings: torsion in $J(R)$} \label{gorrings2}

Let $R$ be a Gorenstein local ring. The main result of this section, Theorem \ref{torsionthm}, is that the class of a module is torsion in $J(R)$ if and only if the module is eventually periodic. This result does not extend verbatim to Cohen-Macaulay local rings; see Example \ref{CMex}. 

In the next lemma, we give a decomposition for the special type of \makebox{$\ztt$-modules} discussed in Remark \ref{ZG}.

\begin{lem}\label{zgmod}
Let $G=\langle t \rangle$, and let $X$ be a set with a $G$-action. Then there is an isomorphism of $\z G$-modules
$$\z X \cong \bigoplus_{n=1}^\infty \left(\frac{\z [t]}{(t^n-1)}\right)^{b_n} \bigoplus \left(\z G\right)^{b_\infty}$$
where $b_\infty, b_n\in \z_{\ge 0}\cup \{\infty\}$ for all $n$. \hfill \qedsymbol
\end{lem}

\bpr
For any $x\in X$, either $t^n x \ne x$ for all $n\ne 0$ and the orbit of $x$ is 
\[Gx = \{t^i x : i\in \z\},\]
or there is an $n>0$ with $t^n x = x$ and 
\[Gx = \{x, tx, t^2x, \dots, t^{n-1}x\}.\]
Then for each $x\in X$ either $\z G x \cong \z G$ or $\z G x \cong \z[t]/(t^n-1)$ as $\z G$-modules; in either case, the map assigning $x$ to $1$ induces an isomorphism. Thus the decomposition of $X$ into orbits gives the desired isomorphism.
\epr

Recall that the \textit{torsion submodule} of a $\z G$-module $L$ is 
$$T_{\z G}(L) = \{u\in L : ru= 0 \tn{ for some } r\in \z G\setminus\{0\}\}.$$
An element $u\in T_{\z G}(L)$ is said to be a \textit{torsion element} of $L$.

\begin{prop} \label{zgmodtor}
An element $u\in \z X$ is torsion if and only if there exists an $n\in\n$ such that $(t^n-1) u = 0$.
\end{prop}

\bpr
Suppose $u$ is torsion in $\z X$. 
Identifying $\z X$ with the right hand side of the isomorphism in Lemma \ref{zgmod}, one finds that $u$ belongs to the submodule $ \oplus_{n=1}^\infty \left(L_n\right)^{b_n}$ of $\z X$ where 
$$\ds L_n = \dfrac{\z [t]}{(t^n-1)}.$$
Consider the case when $u= v+w$ where $v\in L_\ell$ and $w\in L_m$ for some $\ell, m\in \n$.  Then $(t^\ell -1) v = 0$ and $(t^m-1) w = w$, and hence $(t^{m\ell} -1)(v+w) = 0$, since $(t^\ell-1)$ and $(t^m-1)$ both divide $t^{m\ell}-1$.
By induction on the number of terms in $u$, there exists an $n\in \n$ such that $(t^n-1)u = 0$.  

The reverse implication is immediate.
\epr

In light of the preceding results, Theorem \ref{structurethm} has the following corollaries.

\begin{cor} \label{tors_cor}
Let $R$ be a Gorenstein local ring that has the \krs{} property. The following statements hold.
\begin{enumerate}
\item \label{Jtor} 
An element $u\in J(R)$ is torsion if and only if there exists an $n\in \n$ such that $(t^n-1) u = 0$. 
\item \label{tors_elt}
The $\ztt$-module $J(R)$ has nonzero torsion if and only if there is a \fg{} $R$-module $M$ such that $[M]$ is torsion. 
\end{enumerate}
\end{cor}

\bpr
For (\ref{Jtor}), note that $G=\langle t \rangle$ acts on $\mathcal{M}$. By Theorem \ref{structurethm}, 
 $J(R)\cong \mathcal{A} = \z \mathcal{M}$ as $\ztt$-modules. The result then follows from Proposition \ref{zgmodtor}.
 
To prove (\ref{tors_elt}), suppose $u$ is a nonzero torsion element of $J(R)$. In the notation of Lemma \ref{zgmod}, there is an $n\in\n$ such that $b_n\ne 0$. By Theorem \ref{structurethm}, there are some $[M_\alpha]\in \mnor$ that generate $L_n$, and thus $[M_\alpha]$ is torsion for each $\alpha.$ 

The reverse implication is immediate. 
\epr

\subsection{Torsion in $J(R)$}

Let $(R,\m)$ be a local ring. An $R$-module $M$ is said to be \emph{periodic} if there exists an $n\in\n$ such that $M\cong \Omega^n M$. The module $M$ is said to be \emph{eventually periodic} if there exists an $n\in\n$ and $\ell\in \z_{\ge 0}$ such that $\Omega^{\ell} M \cong \Omega^{n+\ell}M$. In either case, the minimal such integer $n$ is called the \emph{period} of $M$.
We make some observations about torsion in $J(R)$ and eventually periodic modules.

\begin{rmk} Let $M$ be a \fg{} $R$-module. It is easy to see that the following statements hold.
\begin{enumerate}
\item [(1)] $[M]$ is torsion in $J(R)$ if and only if $[\Omega^n M]$ is torsion for some (equivalently, all) $n\in \n$. 
\item [(2)] $M$ is eventually periodic if and only if $\Omega^n M$ is eventually periodic for some (equivalently, all) $n\in \n.$
\end{enumerate}
\end{rmk}

In what follows, we write $\widehat{M}$ for the $\m$-adic completion of the $R$-module $M$. 

\begin{lem} \label{complete} 
Let $M$ be a \fg{} $R$-module, and let $i,j\in\z_{\ge 0}$. Then 
\makebox{$\Omega^i_R M \cong \Omega^j_R M$} if and only if $\Omega^i_{\widehat{R}} \widehat{M} \cong \Omega^j_{\widehat{R}} \widehat{M}$.
In particular, $M$ is eventually periodic if and only if the $\widehat{R}$-module $\widehat{M}$ is eventually periodic as an $\widehat{R}$-module. \end{lem}

\bpr
Given that $\Omega^i_R M \cong \Omega^j_R M$, one has 
$$\Omega_{\widehat{R}}^{i} (\widehat{M}) \cong \widehat{\Omega_R^{i}M} \cong \widehat{\Omega_R^{j} M}\cong \Omega_{\widehat{R}}^{j}  (\widehat{M}).$$ 
Suppose that $\Omega^i_{\widehat{R}} \widehat{M} \cong \Omega^j_{\widehat{R}} \widehat{M}$. Then we have the following isomorphisms:
$$\widehat{\Omega_R^i M} \cong \Omega_{\widehat{R}}^{i} (\widehat{M}) \cong \Omega_{\widehat{R}}^{j} (\widehat{M}) \cong \widehat{\Omega_R^{j} M},$$
and so $\Omega_R^i M \cong \Omega_R^{j} M$ by \cite[Cor 1.15]{LW12}. 
\epr

\begin{thm} \label{torsionthm}
Let $R$ be a Gorenstein local ring, and let $M$ be a finitely generated \makebox{$R$-module.} Then $[M]$ is torsion in $J(R)$ with respect to the $\ztt$-action if and only if $M$ is eventually periodic. Moreover, for any $n\in \n$, the following conditions are equivalent:
\begin{enumerate}
\item \label{ntor} $(t^n-1)[M] = 0$ in $J(R)$.
\item \label{npdc} $\Omega^{\ell} M \cong \Omega^{n+\ell} M$ for $\ell\gg 0$.
\end{enumerate}
\end{thm}

\begin{proof}
Suppose $M$ is eventually periodic. Then there are $i,j\in\z_{\ge 0}$ with $i\ne j$ such that $\Omega^i M \cong \Omega^{j}M$. In $J(R)$, $t^{-i} [M] = [\Omega^i M] = [\Omega^{j}M] = t^{-j}[M]$, and hence $(t^{-i}-t^{-j})[M]=0$.

Assume $[M]$ is torsion in $J(R)$. We first show that we can reduce to the case when $R$ is complete with respect to the maximal ideal $\mathfrak{m}$. Let $\widehat{M}$ denote the $\mathfrak{m}$-adic completion of $M$. 
Since the canonical homomorphism $\vp: R\to \widehat{R}$ is flat, Lemma \ref{flat} implies that there is a homomorphism of \makebox{$\ztt$-modules} $J(\vp) \! : \!J(R) \to J(\widehat{R})$ with \makebox{$J(\vp)([M]) = [\,\widehat{M}\,].$} Hence $[M]$ torsion implies that $[\widehat{M}]$ is torsion. 
If the result holds for complete rings, then $\widehat{M}$ is eventually periodic as an $\widehat{R}$-module. Hence Lemma \ref{complete} implies that $M$ is eventually periodic as an $R$-module. 

Assume $R$ is complete. To show that $M$ is eventually periodic, it is enough to show that some syzygy of $M$ is eventually periodic. We may assume $M$ is MCM with no free summands.

Remark \ref{gormcmrmks}.(\ref{dMCM}) implies that $\Omega^d M=0$ or is MCM for $d\gg 0$. If $\Omega^d M =0$, the proof is complete. If not, then replacing $M$ by $\Omega^d M$ we may assume that $M$ is MCM. If $M=N\oplus R$, then $[M]=[N]$ and hence $[M]$ is torsion in $J(R)$ if and only if $[N]$ is torsion. Note that $M$ is eventually periodic if and only if $N$ is eventually periodic, since $\Omega M \cong \Omega N$. Thus we may assume $M$ has no free summands.

As $[M]$ is torsion, Corollary \ref{tors_cor}.(\ref{Jtor}) implies that there is an $n\in\n$ such that $(t^n-1)[M]=0$ in $J(R)$. Proposition \ref{Jmodmcm}.(\ref{cosyzeq}) shows that \makebox{$[\Omega^{-n}M] = t^n[M] = [M]$.} By Remark \ref{gormcmrmks}.(\ref{cosyzMCM}), the $R$-module $\Omega^{-n} M$ is MCM, and thus Corollary \ref{lemma8} implies that \makebox{$\Omega^{-n} M \oplus F \cong M \oplus G$} for some free $R$-modules $F$ and $G$. Then, as $R$-modules, \makebox{$\Omega^n\left(\Omega^{-n} M \oplus F\right) \cong \Omega^n\left( M \oplus G\right)$,} and thus $\Omega^n \Omega^{-n} M  \cong \Omega^n M$. Since $M$ is MCM with no free summands, $M \cong \Omega^n \Omega^{-n} M$ by Remark \ref{gormcmrmks}.(\ref{ddiso}). Hence $M \cong \Omega^n M$, and therefore $M$ is eventually periodic.

It is clear that (\ref{npdc}) implies (\ref{ntor}). The argument in the previous paragraph along with Lemma \ref{complete} shows that (\ref{ntor}) implies (\ref{npdc}).
\end{proof}

\begin{cor} \label{complete2}
Let $R$ be a Gorenstein local ring and $M$ a \fg{} $R$-module. Then $[M]$ is torsion in $J(R)$ with respect to the $\ztt$-action if and only if $[\,\widehat{M}\,]$ is torsion in $J(\widehat{R})$ with respect to the $\ztt$-action.
\end{cor}

\bpr
It follows from Lemma \ref{flat}, and was already used In the proof of Theorem \ref{torsionthm}, that if $[M]$ is torsion in $J(R)$, then $[\widehat{M}]$ is torsion in $J(\widehat{R})$. 

The reverse implication is immediate from Theorem \ref{torsionthm} and Lemma \ref{complete}.
\epr

The following corollary gives the result announced in the abstract. 

\begin{cor} \label{tors_epmod}
Let $R$ be a Gorenstein local ring that has the \krs{} property. The ring $R$ has a periodic module if and only if $J(R)$ has nonzero torsion.
\end{cor}

\bpr
Corollary \ref{tors_cor}.(\ref{tors_elt}) and Theorem \ref{torsionthm} give the desired result.
\epr

\begin{cor} \label{torsioncor}
Suppose $M = \oplus_{i=1}^m M_i$ for some $R$-modules $M_i$. Then $[M]$ is torsion in $J(R)$ if and only if $[M_i]$ is torsion in $J(R)$ for all $i.$
\end{cor}

\bpr
Assume $[M_i]$ is torsion in $J(R)$ for all $i$. For each $i\in \{1,\dots, m\}$, there is an $f_i(t)\in \ztt$ such that $f_i(t)[M_i]=0$ in $J(R)$. Then $f_1(t) \cdots f_m(t) [M] =0$. 

Suppose $[M]$ is torsion in $J(R)$. We first prove the result under the assumption that $R$ is complete. It suffices to consider the case when each $M_i$ is indecomposable. 
Since $[M]$ is torsion in $J(R)$, Theorem \ref{torsionthm} implies that $M$ is eventually periodic. So there is an $n\in\n$ and an $\ell\in \z_{\ge0}$ such that $\Omega^{n+\ell} M \cong \Omega^\ell M,$ and therefore
$$\bigoplus_{i=1}^m \Omega^{n+\ell} (M_i) \cong \bigoplus_{i=1}^m \Omega^\ell (M_i).$$
We prove that each $[M_i]$ is torsion by using induction on $m$, the number of indecomposable summands of $M$. Suppose $M=M_1\oplus M_2$. Then by the \krs{} property, either $\Omega^{n+\ell} (M_i) \cong \Omega^\ell (M_i)$ for $i=1,2$ or $\Omega^{n+\ell} (M_1) \cong \Omega^\ell (M_2)$ and $\Omega^{n+\ell} (M_2) \cong \Omega^\ell (M_1)$. In the first case, it is clear that $M_1$ and $M_2$ are eventually periodic. In the second case, note that 
$$\Omega^{2n+\ell} (M_1) \cong \Omega^{n+\ell} (M_2) \cong \Omega^{\ell} (M_1), $$
and hence $M_1$ is eventually periodic. Similarly, $M_2$ is eventually periodic. Then by Theorem \ref{torsionthm}, $[M_i]$ is torsion in $J(R)$ for $i=1,2$. 

Suppose $M=\oplus_{i=1}^m M_i$ and that the conclusion holds for $s<m$. By the \krs{} property, for each $i$ there exists $j$ such that $\Omega^{n+\ell} (M_i) \cong \Omega^\ell (M_j)$. If there is an $i$ such that $\Omega^{n+\ell} (M_i) \cong \Omega^\ell (M_i)$, then the result follows from the inductive hypothesis. Without loss of generality, suppose $\Omega^{n+\ell} (M_i) \cong \Omega^\ell (M_{i+1})$ for $1 \le i \le m-1$ and $\Omega^{n+\ell} (M_m)\cong \Omega^\ell (M_1)$. The following isomorphisms of $R$-modules show that $M_1$ is eventually periodic:
$$\Omega^{mn+\ell} (M_1)\cong \Omega^{(m-1)n+\ell}(M_2)\cong \cdots \cong \Omega^{n+\ell} (M_m) \cong \Omega^\ell (M_1).$$
Similarly $M_i$ is eventually periodic for $2\le i\le m$, and consequently Theorem \ref{torsionthm} implies that $[M_i]$ is torsion in $J(R)$ for all $i$.

Now suppose that $R$ is any local ring and $[M]$ is torsion in $J(R)$. By Corollary \ref{complete2}, $[\widehat{M}]$ is torsion in $J(\widehat{R})$. Write 
$$\widehat{M} = \bigoplus_{i=1}^m \left(\bigoplus_{j=1}^{a_i} M_{ij}\right) \tn{ with } \hspace{4pt}  \widehat{M}_i = \bigoplus_{j=1}^{a_i} M_{ij}$$
and each $M_{ij}$ an indecomposable $\widehat{R}$-module.
The result for complete rings implies that $[M_{ij}]$ is torsion in $J(\widehat{R})$ for each $i$ and $j$. Then $[\widehat{M}_i]$ is torsion in $J(\widehat{R})$ for all $i$, and so Corollary \ref{complete2} implies that $[M_i]$ is torsion in $J(R)$ \makebox{for all $i$.}
\epr

Theorem \ref{torsionthm} also gives a characterization of hypersurface rings in terms of $J(R)$. The main result of \cite[Thm 7]{DRJ10} is that (\ref{ktor}) implies (\ref{hsurface}) holds when $R$ is an Artinian complete intersection.

\begin{cor} \label{hypersurface}
Let $(R,\m,k)$ be a Gorenstein local ring. Then the following conditions are equivalent:

\begin{enumerate}
\item \label{hsurface} $R$ is a hypersurface;
\item $(1-t^2)\cdot J(R)=0$;
\item $J(R)$ is a torsion module; 
\item \label{ktor} $[k]$ is torsion in $J(R)$ with respect to the $\ztt$-action.
\end{enumerate}
\end{cor}

\bpr
(1) $\Rightarrow$ (2). For any module $M$ over a hypersurface one has $\Omega^{2+\ell} M \cong \Omega^{\ell} M$ for $\ell \gg 0$, by \cite[Thm 6.1]{DE80}, and hence $(1-t^2)[M]=0$. 

(2) $\Rightarrow$ (3) and (3) $\Rightarrow$ (4). These implications are immediate. 

(4) $\Rightarrow$ (1). By Theorem \ref{torsionthm}, the module $k$ is eventually periodic and hence the Betti numbers of $k$ are bounded. By \cite[Cor 1]{TG68}, $R$ is a hypersurface. 
\epr

The following class of examples shows that the statement of Theorem \ref{torsionthm} can fail for non-Gorenstein rings. 

\begin{ex} \label{CMex}
Let $(R,\m,k)$ be a local ring with $\m^2=0$ and embedding dimension $e \ge 2$. Note that $R$ is Cohen-Macaulay but not Gorenstein because the rank of its socle (as a $k$-vector space) is $e$. Then $(1-et) J(R) =0$, but $R$ has no nonzero nonfree eventually periodic module.

First, we note that $k$ is not eventually periodic but $[k]$ is torsion in $J(R)$. Indeed, the sequence
\[0\to \m\to R\to k\to 0\] is exact, and $\Omega k \cong \m$ as $R$-modules. Therefore $\Omega k \cong k^e$, which implies that $k$ is not eventually periodic. On the other hand, 
$t^{-1}[k] = e[k]$ in $J(R)$, and therefore \makebox{$(1-et)[k]=0$.}

Let $M$ be a nonzero, nonfree $R$-module. Since $\m^2=0$, we have $\Omega M \cong k^{\beta_1(M)}$. As $M$ is nonzero, $\beta_1(M) \ge 1$. Since $k$ is not eventually periodic, the module $M$ is not eventually periodic. However, 
\[t^{-1}[M] = [\Omega M] = \beta_1(M)[k]\]
in $J(R)$, and therefore 
\[(1-et)[M] = t(1-et)\beta_1(M)[k]=0.\]

\end{ex}

\begin{rmk}
Using Corollary \ref{tors_cor}.(\ref{Jtor}), one can determine the torsion submodule of $J(R)$ for a Gorenstein local ring that has the \krs{} property:
$$T_{\ztt} (J(R)) = \bigcup_{n=1}^\infty \ann_{J(R)} (1-t^n).$$
If $R$ is a complete intersection, then $T_{\ztt}(J(R)) = \ann_{J(R)}(1-t^2)$ by Theorem \ref{torsionthm}, since \cite[Thm 5.2]{DE80} shows that a periodic module $M$ over a complete intersection has period at most two and hence $(1-t^2)[M] =0$ in $J(R)$. For a Gorenstein ring $R$, however, $[M]$ torsion in $J(R)$ for an $R$-module $M$ need not imply that $(1-t^2)[M]=0$. Indeed, for each $n\in \n$,  there exists an Artinian Gorenstein local ring with a periodic module of period $n$; see \makebox{\cite[Ex 3.6]{GP90}.}
\end{rmk}

\section*{Acknowledgements}
I would like to thank my thesis advisor Srikanth Iyengar for inspiring the development of this work and for his constant support and advice. I am also grateful to Luchezar Avramov and Roger Wiegand for helpful suggestions during the preparation of this paper.

\end{document}